\documentclass[12pt]{amsart}

\usepackage{graphicx}
\usepackage{float}
\usepackage[nice]{nicefrac}
\usepackage{amssymb}
\usepackage{amsmath,amsthm,amscd,amssymb, mathrsfs}

\usepackage{latexsym}

\numberwithin{equation}{section}
\numberwithin{figure}{section}

\theoremstyle{plain}
\newtheorem{theorem}{Theorem}[section]
\newtheorem{lemma}[theorem]{Lemma}

\newtheorem{proposition}[theorem]{Proposition}
\newtheorem{conjecture}[theorem]{Conjecture}

\theoremstyle{definition}
\newtheorem{definition}[theorem]{Definition}

\theoremstyle{remark}
\newtheorem{remark}[theorem]{Remark}

\newtheorem{case[theorem]}{Case}

\def\Ppr{\frak{P}(p,r)}
\def\hull#1{\left \langle#1\right\rangle}
\def\cH{{\mathcal H}}

\def\cV{{\mathcal V}}

\def\({\left(}
\def\){\right)}
\def\[{\left[}
\def\]{\right]}
\def\<{\langle}
\def\>{\rangle}
\def\rank{\mathrm{rank}\,}

\def\mand{\qquad \text{and} \qquad}

\def\F{\mathbb{F}}

\def\R{\mathbb{R}}

\def\A{\mathbb{A}}

\def\P{\mathbb{P}}

\title[Averaging operators over varieties over
 finite fields]{Averaging operators over homogeneous varieties over
 finite fields}

\author{Doowon Koh}

\author{Chun-Yen Shen}

\author{Igor Shparlinski}

\address{Department of Mathematics\\
Chungbuk National University \\
Cheon\-gju city, Chungbuk-Do 361-763 Korea}
\email{koh131@chungbuk.ac.kr}

\address{Department of Mathematics\\ 
National Central University \\
Chungli,  32054 Taiwan}
\email{chunyshen@gmail.com}

\address{Department of Computing 
\\ 
Macquarie University \\
North Ryde, NSW 2109 
Australia}
\email{igor.shparlinski@mq.edu.au}



\begin{document}

\begin{abstract}  In this paper we study the mapping properties of the averaging operator over a variety given by a system of homogeneous equations over a finite field.
We obtain optimal results on the averaging problems over two dimensional varieties whose elements are common solutions of diagonal homogeneous equations.
The proof is based on a careful study of algebraic and geometric properties 
of such varieties. In particular, we show that they  are not contained in any hyperplane and are  complete intersections.
We also address  partial results on averaging problems over  arbitrary dimensional homogeneous varieties which are smooth away from the origin.

\end{abstract} 
\maketitle

\section{Introduction}

\subsection{Motivation} 

Analysis in finite fields is a useful subject because it interacts with other mathematical fields. In addition, the finite field case serves as a typical model for the Euclidean case and possesses structural advantages which enable us to
relate our problems to other well-studied problems in number theory, arithmetic combinatorics, or algebraic geometry.
For these reasons,   problems in Euclidean harmonic analysis have been recently reformulated and studied in the finite field setting.
For example, see~\cite{CSW08, Dv09,  EOT, IK09, LL10,MT04,  Wo99} and references therein. In this paper we investigate $L^p-L^r$ estimates of averaging operators over algebraic varieties given by a system of homogeneous polynomials in finite fields. For Euclidean averaging problems, we refer readers to~\cite{IS96} and~\cite{Li73}. We begin with notation and definitions for averaging problems in finite fields. 
Let $\mathbb F_q^d$ be a $d$-dimensional vector space over a finite field $\mathbb F_q$ with $q$ elements. Throughout this paper, we assume that the characteristic of $\mathbb F_q$ is sufficiently large.  We denote by $dm$ the counting measure on the space $\mathbb F_q^d$.
The pair $(\mathbb F_q^d, dm)$ is named as a function space. We now consider a frequency space, denoted by the pair $(\mathbb F_{q^*}^d, dx)$, where $\mathbb F_{q^*}^d$ and $dx$ denote the dual space of $\mathbb F_q^d$ and  the normalized counting measure on $\mathbb F_{q^*}^d$, respectively. Since $\mathbb F_q^d$ is isomorphic to $\mathbb F_{q^*}^d$ as an abstract group, we  identify $\mathbb F_q^d$ with $\mathbb F_{q*}^d$. For instance, we  write $(\mathbb F_q^d, dx)$ for $(\mathbb F_{q^*}^d, dx)$. This convention helps us to avoid complicated notation appearing in doing some computations.
We shorten both $(\mathbb F_q^d, dm)$ and $(\mathbb F_q^d, dx)$ as just $\mathbb F_q^d$ if there is no risk of confusion between the function space $(\mathbb F_q^d, dm)$ 
and the frequency space $(\mathbb F_q^d, dx)$.
Let $V$ be an algebraic variety in the frequency space $(\mathbb F_q^d, dx)$. We endow $V$ with a normalized surface measure, denoted by $d\sigma$, which can be defined by the relation
$$ \int f(x)~ d\sigma(x)=\frac{1}{|V|} \sum_{x\in V} f(x),$$
where $f:(\mathbb F_q^d, dx)\to \mathbb C$ and $|V|$ denotes the cardinality of $V$. 
Notice that we can replace $d\sigma(x)$ by $q^d|V|^{-1}  V(x) dx$, where $V(x)$ indicates the characteristic function on  $V$. 
Then the convolution function of $f$ and $d\sigma$ is defined on $(\mathbb F_q^d, dx)$:
$$ f\ast d\sigma(y)=\int_{\mathbb F_q^d} f(y-x)~ d\sigma(x)
=\frac{1}{|V|} \sum_{x\in V} f(y-x).
$$
In the finite field setting, the averaging problem is to determine $1\leq p,r\leq \infty$ such that
\begin{equation}\label{defA}\|f\ast d\sigma\|_{L^r(\mathbb F_q^d, dx)} \leq C \|f\|_{L^p(\mathbb F_q^d, dx)} \quad\text{for all}~~f: \mathbb F_q^d \to \mathbb C,\end{equation}
where $C>0$ is independent of the function $f$ and the size of the underlying finite field. 

\begin{definition}
\label{def:Ppr}
We use $\Ppr$ to indicate that inequality~\eqref{defA} holds. 
\end{definition}

As an analogue of averaging problems in Euclidean space, this problem has first been addressed by Carbery, Stones and Wright~\cite{CSW08}.
They mainly investigated the $L^p-L^r$ estimates of the averaging operator over a $k$-dimensional variety given by a vector-valued 
polynomial $P_k: \mathbb F_q^k \to \mathbb F_q^d$.  In particular,
Carbery, Stones, and Wright~\cite{CSW08}  consider a variety 
$\cV_k\subseteq \mathbb F_q^d, 1\leq k\leq d-1$, 
which is given by the range of $P_k: \mathbb F_q^k\to \mathbb F_q^d$ defined by 
$$P_k(t)=\left(t_1,t_2,\ldots, t_k, t_1^2+t_2^2+\cdots+t_k^2, 
\ldots, t_1^{d-k+1}+\cdots+ t_k^{d-k+1}\right)
$$
for
$$
t=(t_1, t_2, \ldots,t_k)\in \mathbb F_q^k.
$$
Observe that the generalized parabolic variety $\cV_k$ can be written by 
\begin{equation}\label{CarberyV} \cV_k=\{x\in \mathbb F_q^d: g_1(x)=g_2(x)=\cdots=g_{d-k}(x)=0\},\end{equation}
where $g_j(x)=x_1^{j+1} +x_2^{j+1}+\cdots+ x_k^{j+1}-x_{k+j}$ for $j=1,2, \ldots, d-k$.
Namely, the variety $\cV_k$ is exactly the collection of  the common solutions of the $d-k$ equations: $g_j(x)=0$ for $j=1,2, \ldots, d-k$.
It is clear that $|\cV_k|=q^k$ for all $k=1,2, \ldots, d-1$, because $x_{k+1}, \ldots, x_d\in \mathbb F_q$ are uniquely determined whenever we choose $x_1,x_2,\ldots, x_k\in \mathbb F_q$.
Applying the  Weil theorem~\cite{We48},  the aforementioned authors~\cite{CSW08} have obtained  the sharp Fourier decay estimates on the variety $\cV_k$ and, 
as a consequence, they give the complete solution of the averaging problem over the variety $\cV_k$. 
Before we present the result of~\cite{CSW08}, we need to introduce one 
more notation:

\begin{definition}
\label{def:ConHul}
For points $P_1, \ldots, P_s \in \R^2$ of the Euclidean plane,
we use $\hull{P_1, \ldots, P_s}$ to denote their {\it convex hull\/}.
\end{definition}

We use Definitions~\ref{def:Ppr} and~\ref{def:ConHul} to formulate 
our main results, in which  $\Ppr$ is related to belonging the 
point $(1/p, 1/r)$ to certain convex polygon. 

We also denote
$$
P_{0,0} = (0,0),\qquad  P_{0,1} = (0,1), \qquad  P_{1,1} =  (1,1).
$$
It is shown in~\cite{CSW08} that
\begin{equation}\label{Carbery} 
\Ppr \iff \(\frac{1}{p}, \frac{1}{r}\)  \in \hull{P_{0,0},P_{0,1},P_{1,1}, \left (\frac{d}{2d-k}, \frac{d-k}{2d-k}\right)}.
\end{equation}
However, if the variety $\cV_k$ is replaced by the homogeneous variety $\cH_k$ defined as
\begin{equation}\label{defH} \cH_k=\{x\in \mathbb F_q^d: h_1(x)=h_2(x)=\cdots=h_{d-k}(x)=0\},\end{equation}
where $h_j(x)=x_1^{j+1} +x_2^{j+1}+\cdots+ x_k^{j+1}-x_{k+j}^{j+1}$ 
for $j=1,2, \ldots, d-k$, then the averaging problem over $\cH_k$ becomes much harder. There are two main reasons why it is difficult to find sharp $L^p-L^r$ averaging estimates over $\cH_k$. First, it is not clear to find the size of $\cH_k$.  Second,  the computation of the Fourier decay estimate on $\cH_k$ is not easy, in part because it can not be obtained by simply applying the Weil theorem~\cite{We48}. Moreover, it may be possible that the Fourier decay on $\cH_k$ is slower  than that on $\cV_k$, because the homogeneous variety $\cH_k$ contains lots of lines which could be key factors to 
make $\cH_k$ flat. These reasons 
suggest  that the $L^p-L^r$ averaging estimates over $\cV_k$ maybe be much better than those over $\cH_k$.  In some cases, it is true but  is not always true in the finite field setting. Indeed, we show here that if $k=2$, then $\cV_k$ and $\cH_k$ yield the same $L^p-L^r$ averaging estimates. 

\subsection{Conjecture on the averaging problem over $\cH_k$} 

The $L^p-L^r$ averaging estimates over $\cH_k$ depend on the maximal dimension of subspaces lying in the variety $\cH_k$. Let us denote by $d\sigma_k$  the normalized surface measure on $\cH_k$.
For a moment, let us assume that $|\cH_k| =(1+o(1)) q^k$ for $k=2,3,\ldots, d-1$, which 
in fact follows from Proposition~\ref{Pro2} and Lemma~\ref{lem:LW} below. 

We recall that  for any real  $U$ and $V$,  $U\lesssim V$ or $V\gtrsim U$ means that there exists $C>0$ independent of $q$ such that $|U|\leq C V$, and $U\asymp V$ is used to indicate that $U\lesssim V$ and $ V\lesssim U$. Throughout the paper, the implied constants may depend on degrees and the number of variables of the polynomials 
defining algebraic varieties under consideration, 
in particular on the integer parameters $d, k, s$.

Suppose that  the following averaging estimate over $\cH_k$ holds true for $1\leq p,r\leq \infty$:
$$
\|f\ast d\sigma_k\|_{L^r(\mathbb F_q^d, dx)} \lesssim \|f\|_{L^p(\mathbb F_q^d, dx)} \quad\text{for all}~~f: \mathbb F_q^d \to \mathbb C.
$$
Then  taking  $f=\delta_0$ as a test function, it follows that 
\begin{equation}\label{necessary1}
\Ppr \Longrightarrow \(\frac{1}{p}, \frac{1}{r}\)  \in 
\hull{P_{0,0},P_{0,1},P_{1,1},\(\frac{d}{2d-k}, \frac{d-k}{2d-k}\right)},
\end{equation}
where $\delta_0(x)=1$ if $x=(0,\ldots,0)$ and $\delta_0(x)=0$ otherwise. 
In fact, this necessary condition for $\Ppr$ has been observed by the authors in~\cite{CSW08} who have also
remarked that if $\cH_k$ contains an $\alpha$-dimensional subspace $\Pi_k$ with $\alpha >k/2$, then the necessary condition~\eqref{necessary1} 
can be improved as 
\begin{equation}\label{imnecessary} 
\Ppr \Longrightarrow \(\frac{1}{p}, \frac{1}{r}\)  \in 
\hull{P_{0,0},P_{0,1},P_{1,1}, Q_{d,k,\alpha}, R_{d,k,\alpha}}, 
\end{equation}
where
\begin{equation*}
\begin{split}
Q_{d,k,\alpha} & = \(\frac{k^2+\alpha d-2\alpha k}{k(d-\alpha)},
\frac{\alpha(d-k)}{k(d-\alpha)}\),\\
R_{d,k,\alpha} & =  
\left(\frac{d(k-\alpha)}{k(d-\alpha)}, \frac{(d-k)(k-\alpha)}{k(d-\alpha)}\right).
\end{split}
\end{equation*}
Hence, to find a more precise necessary condition for $\Ppr$, we need to observe the maximal dimension $\alpha$ of the subspaces $\Pi_k$ lying in the homogeneous variety $\cH_k$.
Since $\alpha$ is always an integer, if $k$ is odd, then it is impossible that $\cH_k$ contains a $k/2$-dimensional subspace $\Pi_k$. 
In this case, if $-1\in \mathbb F_q$ is a square number, then it may happen that $|\Pi_k|=q^{(k+1)/2}$ (namely $\alpha=(k+1)/2$). However,  if $k$ is even, then $\alpha$ can be taken 
as an integer $k/2$.
Combining these observations with~\eqref{necessary1} and~\eqref{imnecessary}, we are lead to the following conjecture (see Figure~\ref{fig}).
\begin{conjecture}\label{conj} For each $k=2,3,\ldots, d-1$, let $\cH_k$ be the homogeneous variety defined as~\eqref{defH}.
\begin{itemize}
\item If $k$ is even, we have
$$ \Ppr \iff \left(\frac{1}{p}, \frac{1}{r}\right) \in 
\hull{P_{0,0},P_{0,1},P_{1,1},\left (\frac{d}{2d-k}, \frac{d-k}{2d-k}\right)}.$$

\item If $k$ is odd and $-1\in \mathbb F_q$ is a square, then 
\begin{equation*}
\Ppr \iff  \left(\frac{1}{p}, \frac{1}{r}\right) \in \hull{P_{0,0},P_{0,1},P_{1,1}, S_{d,k}, T_{d,k}},
\end{equation*}
where
\begin{equation*}
\begin{split}
S_{d,k} & =  \left (\frac{dk-2k+d}{k(2d-k-1)}, \frac{(k+1)(d-k)}{k(2d-k-1)}\right),\\
T_{d,k} & =  \left(\frac{d(k-1)}{k(2d-k-1)}, \frac{(d-k)(k-1)}{k(2d-k-1)}\right).
\end{split}
\end{equation*}
\end{itemize}
\end{conjecture}

 When  $d\geq 3$ is odd and $k=d-1$, it is observed in~\cite{KS11} that the conjecture holds true. 
In the case when $d\geq 4$ is even and $k= d-1$, the conjecture has recently
been established in~\cite{Koh13}. 
Namely, the averaging problem over $H_{d-1}$ has been completely solved where $H_{d-1}=\{x\in \mathbb F_q^d: x_1^2+x_2^2+\ldots+x_{d-1}^2-x_d^2=0\}$ and $-1\in \mathbb F_q$ is a square.
However, there are no known results on the conjecture for $k=2,3,\ldots, d-2$.

\subsection{Statement of main results}

Our first result below says that  Conjecture~\ref{conj} is true for any integer $d\geq 3$ and $k=2$ (see Figure~\ref{fig}).

For each $k=2,3, \ldots, d-1$, let $d\sigma_k$ be the normalized surface measure on  the homogeneous variety $\cH_k\subseteq\mathbb F_q^d$  
given in~\eqref{defH}.

\begin{theorem} \label{main1} If  $d\geq 3$ is an integer and $k=2$, then, 
assuming that the characteristic of $\mathbb F_q$ is sufficiently large,
$$\Ppr \iff \(\frac{1}{p}, \frac{1}{r}\)  \in 
\hull{P_{0,0},P_{0,1},P_{1,1}, \left(\frac{d}{2d-2}, \frac{d-2}{2d-2}\right)}.
$$
\end{theorem}

\begin{remark}\label{rem1}
As mentioned before, this statement  has only been known in~\cite{KS11} for $d=3$ and $k=2$.
Notice from Theorem~\ref{main1} that  the optimal averaging result for the homogeneous variety $\cH_2$ is exactly the same as that in~\eqref{Carbery}  for the general parabolic variety $\cV_2$ defined in~\eqref{CarberyV}.
\end{remark}

Let $\mathbb A^d$ be the affine $d$-space $\overline{\mathbb F}_q^d$, where $\overline{\mathbb F}_q$ denotes the algebraic closure of the finite filed $\mathbb F_q$ with $q$ elements.

For each $k=2,3,\ldots, d-1$, let us  consider the algebraic variety
$$\overline{\cH}_k=\{x\in \mathbb A^d: h_1(x)=h_2(x)=\ldots=h_{d-k}(x)=0\},$$
where  $h_j$, $j=1, \ldots, d-k$, are
 the homogeneous polynomials defined as in~\eqref{defH}.
One interesting point is that the smoothness of $\overline{\cH}_k$ depends on the dimension $d$ of $\mathbb A^d$. Indeed,  we   see from Proposition~\ref{Pro3} below that  the variety $\overline{\cH}_k$ is smooth away from the origin if and only if $d-k=1,2,3$.   In the case when $\overline{\cH}_k$ for $k\geq 3$ is smooth away from the origin,  we are able to obtain certain $L^p-L^r$ averaging estimates on $\cH_k$ (see Figure~\ref{fig}).

Next, we  state our averaging results over $\cH_k$ for $k\geq 3$.

\begin{theorem}\label{main2} If  $\max\{d-3, 3\} \le k \le d-1$, then, 
assuming that the characteristic of $\mathbb F_q$ is sufficiently large, 
$$\(\frac{1}{p}, \frac{1}{r}\) \in \hull{P_{0,0},P_{0,1},P_{1,1}, \left(\frac{d-1}{2d-k-1}, \frac{d-k}{2d-k-1}\right)} \Longrightarrow  \Ppr.
$$
\end{theorem}

\begin{remark} The result of Theorem~\ref{main2} is far from the conjectured averaging result, but in general it gives a sharp $L^p-L^r$ estimate for 
$p=(2d-k-1)/(d-1)$.
\end{remark}

\begin{figure}[H]
\centering
\includegraphics[width=0.8\textwidth]{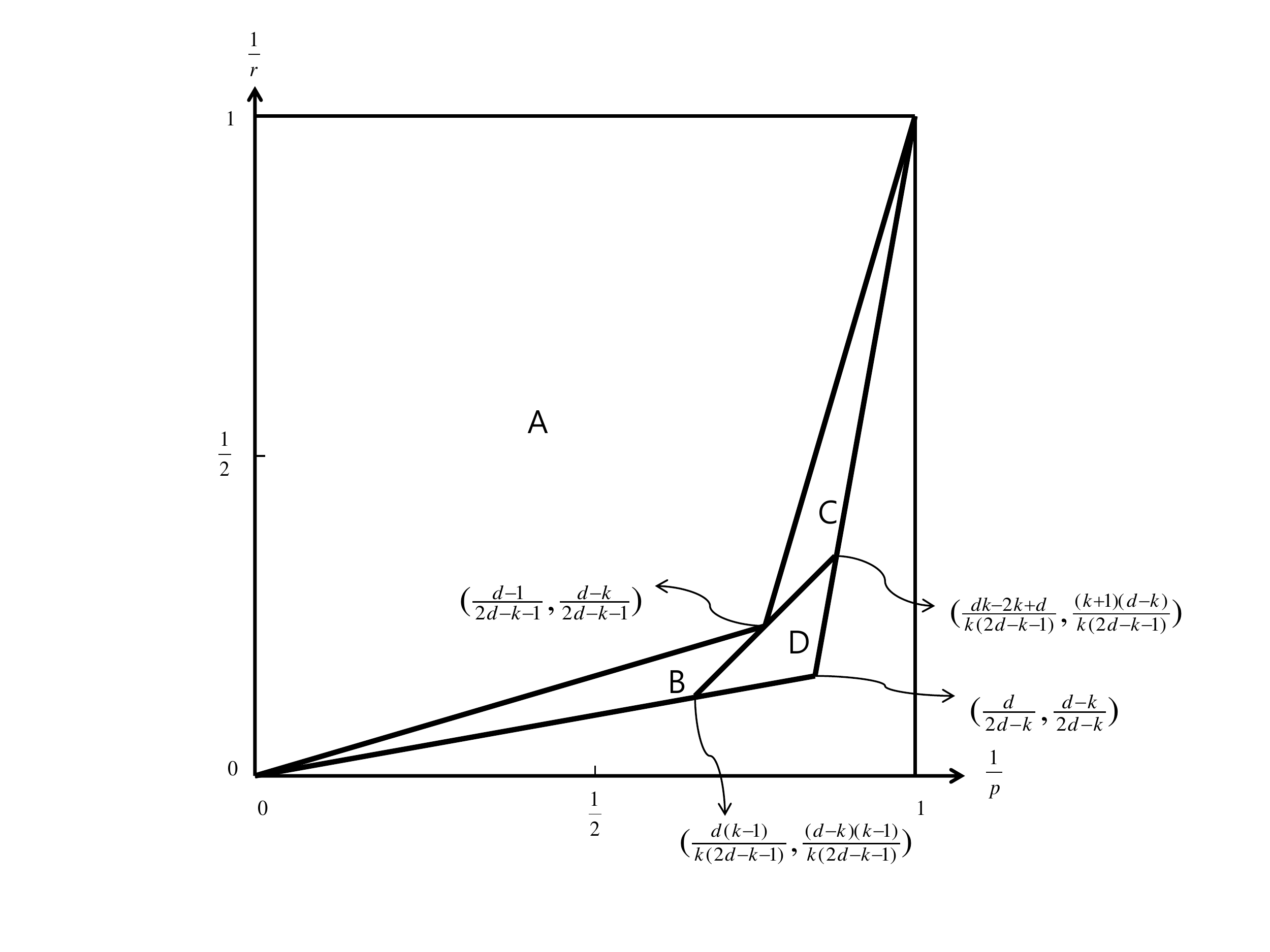}
\caption{When $k\geq 2$ is even,  $A\cup B\cup C\cup D$ is 
region 
of $(1/p, 1/r)$ of Conjecture~\ref{conj}.
Theorem~\ref{main1} covers the conjectured region for $k=2$.
Furthermore, $A\cup B\cup C$ indicates the conjectured region 
in the case when $k\geq 3$ is odd and $-1\in \mathbb F_q$ is a square number.
The region $A$ is corresponding to the result of Theorem~\ref{main2}.}
\label{fig}
\end{figure}

Our work has been mainly motivated by the exponential sum estimates on abstractly given homogeneous varieties due to authors in~\cite{SS90}.
To prove our main results, we  first derive  a useful result about averaging   on general homogeneous varieties with abstract algebraic structures.
Then our main results  follow by applying it to our variety $\cH_k$. 
To do this, we  make the following three key observations on $\cH_k \subseteq\mathbb F_q^d$  for $k=2,3,\ldots, d-1$.

\begin{proposition}\label{Pro1} Suppose that the characteristic of $\mathbb F_q$ is sufficiently large. Then, for every $k=2,\ldots, d-1$, the algebraic variety 
$\overline{\cH}_k\subseteq\mathbb A^d$ is not contained in any hyperplane in $\mathbb A^d$.
\end{proposition}

\begin{proposition}\label{Pro2} Suppose that the characteristic of $\mathbb F_q$ is sufficiently large. 
Then for every $k=2,\ldots, d-1$,   
the algebraic variety $\overline{\cH}_k\subseteq\mathbb A^d$ is  
absolutely irreducible and  $\dim \overline{\cH}_k =k$. 
 \end{proposition}


That is, Propositions~\ref{Pro1} and~\ref{Pro2} assert that $\cH_k$ is a complete intersection.

\begin{proposition} \label{Pro3} 
Suppose that the characteristic of $\mathbb F_q$ is sufficiently large. 
Then for every $k=2,\ldots, d-1$,   
the algebraic variety
$\overline{\cH}_k\subseteq\mathbb A^d$ 
is smooth away from the origin if and only if $d-k=1,2,3$.
\end{proposition}


Furthermore, the smoothness condition on $\cH_2$ is not necessary in completing the proof of Theorem~\ref{main1}.
Therefore,  the conclusion of Theorem~\ref{main1} holds true for any $d$ and $k=2$, and  Conjecture~\ref{conj} for $k=2$ is established.
On the contrary, we  use the smooth condition on $\cH_k$ for $k\geq 3$  in proving Theorem~\ref{main2}.
Thus,  the condition that  $d-k=1,2,3$ is imposed to the statement of Theorem~\ref{main2}.
However, Conjecture~\ref{conj} proposes that  such a smooth condition on $\cH_k$ may  not be important in determining  $L^p-L^r$ averaging estimates over $\cH_k$.
We hope that experts in both harmonic analysis and algebraic geometry may be able to shed insight on the conjecture.

\subsection{Overview of this paper} 

In the remaining parts of this paper, we concentrate on proving Theorem~\ref{main1} and Theorem~\ref{main2} which are our main results.
Instead of proving directly main theorems,  we  derive them by means of working on more general homogeneous varieties with specific geometric structures.

To this end, in Section~\ref{newsec2} we collect facts about the multiplicative character sums and the existence of a primitive prime divisor of a family of
shifted monomials. In particular, we make use of a polynomial analogue
of the Zsigmondy theorem which is due to Flatters and Ward~\cite{FW11}. 

Section~\ref{sec2} is devoted to setting up notation and basic concepts  essential in defining  abstract varieties in algebraic geometry.

In Section~\ref{sec3},  we  derive a result  for averaging  problems over  general homogeneous varieties, where  we adapt the standard analysis technique in~\cite{CSW08} 
together with the results on exponential sums in~\cite{SS90},  see Lemma~\ref{lem:HomVar} below. 
In fact, this result generalizes our main results related to $\cH_k$. 

In Section~\ref{sec4}, we show that Lemma~\ref{lem:HomVar} applies to the variety $\cH_k$ and complete the proofs of our main results, 
that is, Theorems~\ref{main1} and~\ref{main2}.

We note that our main tool are bounds of exponential sums along algebraic varieties, which 
we interpret as results about the decay of Fourier coefficients. 

\section{Multiplicative character sums and  roots of some polynomials}
\label{newsec2}

\subsection{Root of shifted monomials}
We need the following simple observation, which immediately follows 
from the Taylor formula (which applies if the characteristic is large enough). 

\begin{lemma} \label{endlem1}
For any fixed integer $s\ge 1$, if the characteristic of $\mathbb F_q$ is sufficiently large 
then for any $a \in \mathbb F_q^*$, the polynomial $t^s+a \in \mathbb F_q[t]$ 
has no multiple roots.
\end{lemma}


\begin{lemma} \label{endlem15}
For any fixed integer $s\ge 1$, if the characteristic of $\mathbb F_q$ is sufficiently large 
then  the polynomial $t^s+1 \in \mathbb F_q[t]$ has at least one root which is
 not a root of the polynomials $t^j+1 \in \mathbb F_q[t]$, $j = 1, \ldots, s-1$. 
\end{lemma}

 \begin{proof}  By a result of Flatters and Ward~\cite[Theorem~2.6]{FW11}, 
if the characteristic of $\mathbb F_q$ is large enough then
$t^{2s}-1$ has an irreducible factor $q(t) \in \mathbb F_q[t]$ that does not divide 
any of the polynomials $t^{j}-1 \in \mathbb F_q[t]$, $j = 1, \ldots, 2s-1$.
In particular, $q(t)$ is relatively prime with $t^s-1$ and thus is 
a divisor of 
$$
 t^{s}+1 = \frac{t^{2s}-1}{t^{s}-1}.
$$
Furthermore, $q(t)$ is relatively prime to 
$$
 t^{j}+1 = \frac{t^{2j}-1}{t^{j}-1}, \qquad j = 1, \ldots, s-1, 
$$
which concludes the proof. 
\end{proof}

\subsection{Multiplicative character sums  with shifted monomials}
We also need the following result due to Wan~\cite[Corollary~2.3]{Wa97}
that follows almost instantly from the Weil bound in the form given 
in~\cite[Theorem~11.23]{IwKow04}.

\begin{lemma}\label{endlem2} Let $g_1(t),\ldots, g_s(t)$ 
be $s$ monic pairwise prime polynomials in $\mathbb F_q[t]$.
Denote by $\chi_1,  \ldots, \chi_s$  nontrivial multiplicative characters 
of $\mathbb F_q$ with order $d_1, \ldots, d_s$, respectively. 
If for some $i=1,2, \ldots, s$, the polynomial $g_i(t)$ is not of the form $q(t)^{d_i}$ 
with  $q(t) \in \mathbb F_q[t]$, then we have
$$ \left| \sum_{t\in \mathbb F_q} \chi_1(g_1(t)) \cdots \chi_s(g_s(t)) \right| \lesssim q^{\frac{1}{2}}.$$
\end{lemma}

 \begin{lemma} \label{endlem3}
For any fixed integer $s\ge 1$, if the characteristic of $\mathbb F_q$ is sufficiently large, 
then for any multiplicative characters $\chi_j$, $j=1,2, \ldots, s$, 
among which at least one is nontrivial, we have 
$$\left|\sum_{t\in \mathbb F_q} \prod_{j=1}^{s}
\chi_j (t^{j+1}+1)\right|\lesssim q^{\frac{1}{2}}.$$
\end{lemma}

 \begin{proof} 
After ignoring all trivial characters, it suffices to prove that for some
positive integer  $m \le s$,  we have
\begin{equation}\label{conquer} \left|\sum_{t\in \mathbb F_q} \widetilde{\chi}_1(t^{s_1+1} +1)\cdots
\widetilde{\chi}_m(t^{s_m+1} +1)\right|\lesssim q^{\frac{1}{2}},\end{equation}
where $1\leq s_1<s_2<\ldots<s_m\leq s$ and  $\tilde{\chi}_1,  \ldots,\tilde{\chi}_m$ denote  nontrivial multiplicative characters of $\mathbb F_q$.
Factoring the polynomials $t^{s_i+1} +1$,  $i=1, \ldots,m$, into irreducible factors
over $\F_q$  and using the multiplicativity, we see from 
Lemma ~\ref{endlem1} that 
$$
\sum_{t\in \mathbb F_q} \widetilde{\chi}_1(t^{s_1+1} +1)\cdots
\widetilde{\chi}_m(t^{s_m+1} +1)
= \sum_{t\in \mathbb F_q} {\eta}_1\(q_1(t)\) \cdots
{\eta}_u(q_u(t))
$$
for some multiplicative characters $\eta_i$ and monic pairwise prime 
polynomials $q_i$, $i=1,2,\ldots, u$. 

Furthermore, by Lemma~\ref{endlem15} we have $\eta_{i_0}= \widetilde{\chi}_{i_0}$
for at least one $i_0 \in \{1, \ldots, u\}$, and thus $\eta_{i_0}$ is a nontrivial 
character. Now using Lemma~\ref{endlem2}  we complete the proof.  
\end{proof}

\section{Algebraic properties of general homogeneous varieties}
\label{sec2}

\subsection{Preliminaries}
In this section, we review known facts on general varieties generated by a system of $s-$homogeneous polynomials in $\mathbb F_q[x_1,x_2, \ldots, x_d]$.
We begin by setting up notation. 

Let $2\leq s\leq d-1$ be an integer. 
Assume we are given $s-$homogeneous polynomials in $d$ variables over $\F_q$ 
of  degree  at least two each,  which we write as 
$$
f_j(x)\in \mathbb F_q[x_1,x_2,\ldots,x_d],  \quad 
\deg f_j \ge 2, \qquad j =1, \ldots, s, 
$$
where $x= (x_1,x_2,\ldots,x_d)$.
Now, define the closed algebraic set 
\begin{equation}
\label{eq:HA-bar}
\overline H_\A=\{x\in \mathbb A^d~:~f_1(x)=f_2(x)=\ldots=f_s(x)=0\}.
\end{equation}
Let $H_\A$ be the collection of points in $\overline H_\A$ with coordinates 
in $\mathbb F_q:$
\begin{equation}
\label{eq:HA} 
H_\A=\{x\in \mathbb F_q^d~:~f_1(x)=f_2(x)=\ldots=f_s(x)=0\}.
\end{equation}
 We also use the standard notation $\mathbb P^{d-1}$ for the $(d-1)$-dimensional projective space over $\overline{\mathbb F}_q$, which can be considered as the collection of all one dimensional subspaces of the vector space $\mathbb A^d$. For $P=[a_1:a_2:\ldots:a_d]\in \mathbb P^{d-1}$ and a polynomial $f\in \overline{\mathbb F}_q[x_1,\ldots,x_d]$,
recall that $f(P)=0$ means that $f(\lambda a_1, \ldots, \lambda a_d)=0$ for all $\lambda\neq 0$.
Like the algebraic subset $\overline H_\A$ of the affine space $\mathbb A^d$,  we define the projective algebraic set
$$
\overline H_\P=\{P\in \mathbb P^{d-1}~:~f_1(P)=f_2(P)=\ldots=f_s(P)=0\}.
$$
Let us recall an affine cone over a projective subset in $\mathbb P^{d-1}$.
Denote by $\pi: \mathbb A^d\setminus \{(0,\ldots,0)\} \to \mathbb P^{d-1}$ the projection map defined by
$$\pi(x_1, \ldots, x_d)=[x_1:\cdots:x_d].$$
Then the affine cone over $Y\subseteq\mathbb P^{d-1}$ is defined by
$$C(Y)=\pi^{-1}(Y) \cup \{(0,\ldots,0)\} \subseteq\mathbb A^d.$$
Notice that  $\overline H_\A$ is the affine cone over the 
projective variety $\overline H_\P$.

\begin{definition}\label{def1} We say that a homogeneous variety 
$H_\A\subseteq\mathbb F_q^d$ defined as in~\eqref{eq:HA} is a {\bf complete intersection\/}  if the following two conditions hold:
\begin{itemize}
\item $\overline H_\A\subseteq\mathbb A^d$ is an affine cone over a 
projective variety $\overline H_\P$ which is not contained in a hyperplane, 
\item  $\overline H_\A$ is an absolutely irreducible variety of dimension $d-s$ (or $\dim\overline H_\P=d-1-s$).
\end{itemize}
\end{definition}

\begin{definition}\label{def2}
We say that a homogeneous variety 
$H_\A\subseteq\mathbb F_q^d$ defined as in~\eqref{eq:HA} is   {\bf smooth\/} 
if $\overline H_\A$ is smooth away from the origin. 
\end{definition}

\subsection{Exponential sums and Fourier coefficients} 

It is observed in~\cite{SS90} that if $H_\A$ is a smooth homogeneous variety, which is a complete intersection, then $\overline H_\A \cap \overline{\Pi}(m)$ can have at most isolated singularities in $\mathbb P^{d-1}$ where 
$$
\overline{\Pi}(m)=\{x\in \mathbb A^d: m\cdot x=0\}$$  
for $m\neq(0,\ldots,0)$, where $m\cdot x$ denotes the inner products of the vectors $m$ and $x$.

 From this observation, they state and prove  the following exponential sum estimates on $H_\A$ (see~\cite[Theorem~1]{SS90}).
\begin{lemma}\label{Igor} 
Let  $H_\A \subseteq\mathbb F_q^d$  be defined as in~\eqref{eq:HA}.
Suppose that $H_\A$ is a {\bf smooth} homogeneous variety, which is a complete intersection,  and the characteristic of $\mathbb F_q$ is sufficiently large.
Then we have for all $m\in \mathbb F_q^d\setminus \{(0,\ldots,0)\}$ 
$$\left| \sum_{x\in H_\A} \psi(m\cdot x)\right|   \lesssim \left\{\begin{array}{ll} 
q^{(d-s+1)/2} \quad &\text{if}~~ d-s\geq 3,\\                                                                                                                                q\quad &\text{if}~~ d-s= 2, \end{array} \right.$$
where $\psi$ denotes a nontrivial additive character of $\mathbb F_q$. 
\end{lemma}

Here, we point out that the proof of Lemma~\ref{Igor} for $d-s=2$ was given by  the authors in~\cite{SS90} without using the smoothness assumption on $H_\A$.  
Therefore,  the smoothness condition on $H_\A$ can be relaxed for $d-s=2$.
 Indeed, the following bound follows immediately from a result of Cochrane~\cite[Theorem~4.3.5]{Co94}. 
  
\begin{lemma}\label{Igor1} If $H_\A \subseteq\mathbb F_q^d$ is a homogeneous variety
which is a complete intersection given by~\eqref{eq:HA} 
of dimension $\dim \overline H_\A =d-s$, where $\overline H_\A$ 
is given by~\eqref{eq:HA-bar}, then
$$  \left|\sum_{x\in H_\A} \psi(m\cdot x)\right|  \lesssim q^{d-s-1}
$$
for all $m\in \mathbb F_q^d\setminus\{(0,\ldots,0)\}$.
\end{lemma}

The following estimate on the cardinality of $H_\A$  due to 
Chatzidakis,  van den Dries and  Macintyre~\cite[Proposition~3.3]{CvdDM}   
gives an extension of the result of 
Lang and Weil~\cite{LW54}. 

\begin{lemma}
\label{lem:LW}
 Suppose that $ \overline{V} \subseteq\mathbb A^d$ is an algebraic variety 
 with $\nu$ absolutely irreducible components and of dimension $e$ defined by polynomials over $\F_q$ and let 
 $V=\{x\in \overline{V} \cap \mathbb F_q^d\}$. Then 
$$  |V| -\nu q^e  \lesssim q^{e-1/2}.
$$
\end{lemma}

It is clear from  Lemma~\ref{lem:LW} that  
$|H_\A| =(1+ o(1)) q^{d-s}$ if $H_\A$ is a homogeneous variety, given by~\eqref{eq:HA} which 
is a complete intersection in $\mathbb F_q^d$.

Now, we endow a homogeneous variety $H_\A$ with the normalized surface measure $d\sigma_H$.
Recall that if $f: (\mathbb F_q^d, dx)\to \mathbb C$, then 
$$ 
\int f(x)~ d\sigma_H(x)=\frac{1}{|H_\A|} \sum_{x\in H_\A} f(x).
$$
The following decay estimates of the  Fourier coefficients
$$
(d\sigma_H)^\vee(m)= \frac{1}{|H_\A|} \sum_{x\in H_\A} \psi(m\cdot x), 
\qquad m\in \mathbb F_q^d,
$$
 on $H_\A$ follow 
immediately from Lemma~\ref{Igor1} and~\ref{Igor}.

\begin{lemma}\label{maindecay} Let $d\sigma_H$ be the normalized surface measure on the homogeneous variety $H_\A \subseteq\mathbb F_q^d$ given by~\eqref{eq:HA}, which is a complete intersection.  If the characteristic of $\mathbb F_q$ is sufficiently large, 
then: 
\begin{itemize}
\item[{\bf (i)}] If  $d-s=2$, then we have
$$ |(d\sigma_H)^\vee(m)|\lesssim q^{-1}$$
for all $m\in \mathbb F_q^d\setminus\{(0,\ldots,0)\}$. 
\item[{\bf (ii)}] If $H_\A$ is  smooth and $d-s\geq 3$, then 
$$|(d\sigma_H)^\vee(m)|\lesssim q^{-(d-s-1)/2}
$$
for all $m\in \mathbb F_q^d\setminus\{(0,\ldots,0)\}$. 
\end{itemize}

\end{lemma}

\section{Fourier coefficients and $L^p-L^r$ averaging estimates over $H_\A$} \label{sec3}

\subsection{Estimates for varieties with given rate of decay of Fourier coefficients}
First we need the following general result which can be obtained by adapting the arguments in~\cite{CSW08}. 
For the sake of completeness, we provide the proof in full detail. 

\begin{lemma}\label{formula} Let $d\sigma_H$ be the normalized surface measure on an affine homogeneous variety $H_\A\subseteq\mathbb F_q^d$ given by~\eqref{eq:HA}, which is 
a complete intersection.
If $(d\sigma)^\vee(m)\lesssim q^{-\vartheta/2}$ for all $m\in \mathbb F_q^d\setminus\{(0,\ldots,0)\}$ and for some fixed $\vartheta >0$, then  $\frak{P}(p,r)$
holds with 
$$p =\frac{2s+\vartheta}{s+\vartheta} \mand r = \frac{2s+\vartheta}{s}.
$$
\end{lemma}

\begin{proof}
We must show that  
$$\|f\ast d\sigma_H\|_{L^{r}({\mathbb F_q^d},dx)}\lesssim \|f\|_{{L^p}({\mathbb F_q^d},dx)}
$$
for all function $f: (\mathbb F_q^d, dx) \to \mathbb C$ 
with the above values of $p$ and $r$. 

 Define a function $K$ on $({\mathbb F_q^d}, dm)$ by $K=(d\sigma_H)^\vee -\delta_0$.
Observe that $d\sigma_H(x)= \widehat{K}(x)+\widehat{\delta_0}(x)= \widehat{K}(x)+1$  for $x\in (\mathbb F_q^d, dx)$, 
where for a function $F$ on $({\mathbb F_q^d}, dm)$ we define
$$
\widehat{F}(x) =  \sum_{m \in \F_q^d} F(m) \psi(-m\cdot x).
$$
and, as before,  $m\cdot x$ denotes the inner products of $m$ and $x$. 
Since $\vartheta >0$ and $dx$ is the normalized counting measure on $\mathbb F_q^d$,  it 
follows from  Young's inequality (see~\cite{Gr03, Jo01}) that 
 $$\|f\ast 1\|_{L^r({\mathbb F_q^d},dx)}\lesssim \|f\|_{L^p({\mathbb F_q^d},dx)}.$$ 
 Thus, it suffices to prove that 
\begin{equation}
\label{ref}
\|f\ast \widehat{K}\|_{L^r({\mathbb F_q^d},dx)}\lesssim \|f\|_{L^p({\mathbb F_q^d},dx)}
\end{equation}
for all functions $f:(\mathbb F_q^d, dx) \to \mathbb C$.

Notice that~\eqref{ref} can be obtained by interpolating 
\begin{equation}\label{first1}
\|f\ast \widehat{K}\|_{L^2({\mathbb F_q^d},dx)}\lesssim  q^{-\vartheta/2}\|f\|_{L^2({\mathbb F_q^d},dx)}
\end{equation}
and 
\begin{equation}\label{second2}
\|f\ast \widehat{K}\|_{L^{\infty}({\mathbb F_q^d},dx)}\lesssim  q^s \|f\|_{L^1({\mathbb F_q^d},dx)}.
\end{equation}
It remains to prove~\eqref{first1} and~\eqref{second2}. By the definition of $K$ and the assumption that $(d\sigma_H)^\vee(m)\lesssim q^{-\vartheta/2}$ for $m\neq (0,\ldots,0)$, we see that
$$ \max\limits_{m\in \mathbb F_q^d} |K(m)| \lesssim q^{-\vartheta/2}.
$$ 

Therefore,~\eqref{first1} follows by applying the Plancherel theorem (see~\cite{Gr03, Jo01}):
\begin{equation*}
\begin{split}
\|f\ast \widehat{K}\|_{L^2({\mathbb F_q^d},dx)} 
& = \|f^\vee K\|_{L^2({\mathbb F_q^d},dm)} \\
& \lesssim q^{-\vartheta/2} \|f^\vee\|_{L^2({\mathbb F_q^d},dm)}
= q^{-\vartheta/2} \|f\|_{L^2({\mathbb F_q^d},dx)}.
\end{split}
\end{equation*}
To prove~\eqref{second2}, notice that $|H_\A| =(1+ o(1)) q^{d-s}$, because $H_\A$ is a complete intersection. 
From Young's inequality and the observation that 
$\|\widehat{K}\|_{L^\infty({\mathbb F_q^d},dx)}\lesssim q^s$, we obtain~\eqref{second2}:
$$ \|f\ast \widehat{K}\|_{L^{\infty}({\mathbb F_q^d},dx)} \leq \|\widehat{K}\|_{L^\infty({\mathbb F_q^d},dx)}\|f\|_{L^1({\mathbb F_q^d},dx)}\lesssim  q^s \|f\|_{L^1({\mathbb F_q^d},dx)}.$$
Thus,  the proof of Lemma~\ref{formula} is complete.
\end{proof}

\subsection{Main estimates} 

As a direct application of the Fourier decay estimates in Lemma~\ref{maindecay}, we can now 
derive averaging results related to general homogeneous variety $H_\A$, which is a complete intersection.
Applying Lemma~\ref{formula} with Lemma~\ref{maindecay}  yields the result below.

\begin{lemma}\label{lem:HomVar} Let $d\sigma_H$ be the normalized surface measure on a homogeneous variety $H_\A \subseteq \mathbb F_q^d$, given by~\eqref{eq:HA}, 
which is a complete intersection.  If the characteristic of $\mathbb F_q$ is sufficiently large, then:
\begin{itemize}
\item[{\bf (i)}] If $d-s=2$, then 
$$ \Ppr  \iff \(\frac{1}{p}, \frac{1}{r}\) \in \hull{
P_{0,0},P_{0,1},P_{1,1},\left(\frac{d}{2d-2}, \frac{d-2}{2d-2}\right)}.
$$
\item[{\bf (ii)}] If $H_\A$ is a   smooth  and $d-s\geq 3$, then 
$$
\(\frac{1}{p}, \frac{1}{r}\)  \in \hull{P_{0,0},P_{0,1},P_{1,1},\left(\frac{d-1}{d+s-1}, \frac{s}{d+s-1}\right)} \Longrightarrow \Ppr.
$$ 

\end{itemize}
\end{lemma}

\begin{proof}
To prove~{\bf{(i)}}, let us assume that $d-s=2$. Then $|H_\A|=(1+ o(1)) q^{d-s}= q^2$, because $H_\A$ is a complete intersection. 
Now, suppose that $\Ppr$.  In particular, we see that  
$$
q^{-(d+2r-2)/r}\asymp\|\delta_0\ast d\sigma_H\|_{L^r(\mathbb F_q^d, dx)}\lesssim \|\delta_0\|_{L^p(\mathbb F_q^d, dx)}=q^{-d/p}.
$$ 
It therefore  follows that 
$$\frac{-d-2r+2}{r}\leq \frac{-d}{p}.
$$ 
By duality, we also have 
$$
\frac{-d-2p^*+2}{p^*}\leq \frac{-d}{r^*},
$$ 
where 
$$p^* = \frac{p}{p-1}\qquad \text{and} \qquad r^*= \frac{r}{r-1}
$$ 
denote the H\"{o}lder conjugates
 of $p$ and $r$, respectively.
In conclusion, 
\begin{equation}
\label{eq:1p1r}
\(\frac{1}{p}, \frac{1}{r}\) \in \hull{P_{0,0},P_{0,1},P_{1,1},\left(\frac{d}{2d-2}, \frac{d-2}{2d-2}\right)}.  
\end{equation}

Conversely, we now assume that the inclusion~\eqref{eq:1p1r} holds.
If $1\leq r \leq p\leq \infty$, then it is clear that $\Ppr$, because both $d\sigma_H$ and $({\mathbb F_{q}^d},dx)$ have total mass $1$.
By the interpolation theorem, it therefore suffices to prove that 
$$
\frak{P}\left(\frac{2d-2}{d}, \frac{2d-2}{d-2}\right)
$$
holds.  Since $d-s=2$,  applying Lemma~\ref{formula} with  
Lemma~\ref{maindecay}~{\bf{(i)}} yields the above property, and the the proof of Lemma~\ref{lem:HomVar}~{\bf{(i)}} is complete.

In order to prove~{\bf{(ii)}}, it is enough to show that 
 $$\frak{P}\left(\frac{d+s-1}{d-1},\frac{d+s-1}{s}\right) $$
 holds. 
However, this follows immediately by using Lemma~\ref{formula} together with 
Lemma~\ref{maindecay}~{\bf{(ii)}}.
\end{proof}

\begin{remark}Even if Lemma~\ref{lem:HomVar} provides us of powerful averaging results on general homogeneous varieties,  
applying it in practice may not be simple, because it contains certain abstract hypotheses.
\end{remark}

\section{Proofs of Main Results}
\label{sec4}

\subsection{Preliminaries} 
 In this section, we complete the proofs of Theorem~\ref{main1} and Theorem~\ref{main2} which are considered as main theorems in this paper. 
We complete the proofs by showing that Lemma~\ref{lem:HomVar} is a general version of  both Theorem~\ref{main1} and Theorem~\ref{main2}.
To do this, we begin by recalling  from~\eqref{defH} that for each $k=2,3,\ldots, d-1$, our homogeneous variety $\cH_k$ is exactly the common solutions in $\mathbb F_q^d$ of a system of the $(d-k)$ equations
\begin{equation}\label{system} 
\begin{array}{ll} &h_1(x)= x_1^2+x_2^2 +\cdots+ x_k^2 -x_{k+1}^2=0,\\
&\qquad \qquad \qquad \qquad \vdots\\
& h_j(x)= x_1^{j+1}+x_2^{j+1} +\cdots+ x_k^{j+1} -x_{k+j}^{j+1}=0,\\
&\qquad \qquad \qquad \qquad \vdots\\
 &h_{d-k}(x)= x_1^{d-k+1}+x_2^{d-k+1} +\cdots+ x_k^{d-k+1} -x_{d}^{d-k+1}=0,\\
\end{array}, 
\end{equation}
where $j=1,2,\ldots, (d-k)$. Let $s=d-k$ which is the number of homogeneous equations $h_j$ defining $\cH_k$. Then it is clear that $\overline{\cH}_k$ is an  affine cone over its corresponding projective variety determined by $s$-homogeneous polynomials $h_j$.  Thus, if we are able to show that the conclusions 
of Propositions~\ref{Pro1} and~\ref{Pro2} 
hold for $k=2$ then  Theorem~\ref{main1} follows from 
Lemma~\ref{lem:HomVar}~{\bf{(i)}}. Furthermore, if all  
Propositions~\ref{Pro1}--\ref{Pro3} hold true for $k\geq 3$, then Theorem~\ref{main2}  follows from Lemma~\ref{lem:HomVar}~{\bf{(ii)}} where the smoothness condition on $\overline{\cH}_k$ is essential.
In summary, to prove both Theorem~\ref{main1} and Theorem~\ref{main2}, it suffices to justify Propositions~\ref{Pro1}--\ref{Pro3}.

\subsection{Proof of Proposition~\ref{Pro1}} 

 Since any hyperplane in $\mathbb A^d$ is a subspace with dimension $d-1$, it suffices to prove that 
there exists $d$ linearly  independent points $\{P_1, P_2, \ldots, P_d\} \subseteq\overline \cH_k$.                                              
Now fix $k=2,3,\ldots, (d-1)$. For each $j=1,2,\ldots, d-k$,  choose a $\beta_{k+j}\in \mathbb A$ such that $\beta
_{k+j}^{j+1}=1$ and $\beta_{k+j}\neq 1$.
Since $\mathbb A$ is an algebraic closure and the characteristic of $\mathbb F_q$ is sufficiently large,  the $\beta_{k+j}$ always exists.
Denote by $I_{k\times k}$ the $k\times k$ identity matrix.
We also define $1_{k\times (d-k)}$ as the $k\times (d-k)$ matrix whose all entries are $1$.
Also define the following $(d-k)\times k$ matrix $C_{(d-k)\times k}$ and the $(d-k)\times (d-k)$ matrix $ D_{(d-k)\times (d-k)}$: 
\begin{equation*}
\begin{split}
& C_{(d-k)\times k} = 
\left[\begin{matrix}
      0 & 0 & \cdots & 0& 1\\
      0 & 0& \cdots & 0&1\\
       \vdots&\vdots& \vdots&\vdots&\vdots \\
      0 & 0 & \cdots & 0&1\\
      0 & 0 & \cdots & 0&1 \\
\end{matrix}\right], \\
&D_{(d-k)\times (d-k)}  = 
\left[\begin{matrix}
      \beta_{k+1} & 1 & \cdots & 1& 1\\
      1 & \beta_{k+2}&1 & \cdots &1\\
       \vdots&\vdots& \vdots&\vdots&\vdots \\
     1 &  \cdots&1 & \beta_{d-1}&1\\
      1& 1&\cdots & 1&\beta_d 
\end{matrix}\right].
\end{split}
\end{equation*}

Now consider the $d\times d$ matrix $M_{d\times d}$ defined by
$$M_{d\times d}=\left[\begin{matrix}P_1\\
                                            P_2\\
                                             \vdots\\
                                             P_d \end{matrix}\right]
=\left[ \begin{matrix}
                                                 I_{k\times k}& 1_{k\times (d-k)}\\
                                                 C_{(d-k)\times k} & D_{(d-k)\times(d- k)} \end{matrix}\right].$$
Note that all $P_1,P_2, \ldots, P_d$ are solutions of a system of 
equations~\eqref{system}. Hence, it follows that
$\{P_1,P_2, \ldots, P_d\} \subseteq\overline{\cH}_k$ for any $k=2,3,\ldots, (d-1)$. 
Moreover, since $\beta_{k+j}-1\neq 0$ for all $j=1,2, \ldots, d-k$,  it follows from simple Gauss elimination that
the rank of the matrix $ M_{d\times d}$ is exactly $d$, which  completes the proof of Proposition~\ref{Pro1}.

\subsection{Proof of Proposition~\ref{Pro2}} 
Recall from~\eqref{system}  that  $\overline{\cH}_k \subseteq\mathbb A^d$ is given by a system of $(d-k)$ homogeneous equations.
 For each $j=1,2, \ldots, d-k$, define an algebraic set
$$ \overline{\cH}_k^j=\{x\in \mathbb A^d~:~h_j(x)=0\},$$
where $h_j$ is defined by~\eqref{system}.
By the definition of $\overline{\cH}_k$, it follows that
\begin{equation}\label{one}\overline{\cH}_k=\bigcap_{j=1}^{d-k} \overline{\cH}_k^j.
\end{equation}
We need the following claim.

\begin{lemma}\label{claim1} For each $n\in \{1,2,\ldots, (d-k-1)\}$, we have
$$ \left(\bigcap_{j=1}^n \overline{\cH}_k^j\right) \bigcap \overline{\cH}_k^{n+1}
\neq \emptyset $$
and there exists $\alpha\in \mathbb A^d$ such that
$$ \quad \alpha \in \bigcap_{j=1}^n \overline{\cH}_k^j \quad \text{and}\quad \alpha\notin \overline{\cH}_k^{n+1}.$$
\end{lemma}

\begin{proof}
The first part  is trivial.
For the second part of this claim, fix $n\in \{1,2,\ldots, (d-k-1)\}$ and let $l\in \overline{\mathbb F}_q$ with $l^{n+2}\neq 1$.
Now, choose an $\alpha=(\alpha_1, \alpha_2,\ldots, \alpha_d)\in \mathbb A^d$ whose coordinates satisfy that
$$\alpha_j=\left\{\begin{array}{ll}0\quad&\text{if}~~ j=2,3,\ldots, k,\\
 l \quad&\text{if}~~ j=k+n+1,\\
1     \quad&\text{otherwise.}\end{array}\right.
$$ 
Then it is straightforward to check that $\alpha\in \bigcap_{j=1}^n \overline{\cH}_k^j \quad \text{and}\quad \alpha\notin \overline{\cH}_k^{n+1}$ and the result follows. 
\end{proof}

 To compute the dimension of $\overline{\cH}_k$, we apply the following result,
 see~\cite[Page~55]{Ga02} for a proof.  

\begin{lemma}\label{dimlem} Let $\overline{V} \subseteq\mathbb A^d$ be an irreducible algebraic set, and let $f\in \overline{\mathbb F}_q[x_1,x_2,\ldots, x_d]$ be a nonconstant polynomial which does not vanish identically on $\overline{V}$.  In addition, let us define $ \mathbb Z(f)=\{x\in \mathbb A^d~:~f(x)=0\}$. If $\overline{V} \cap \mathbb Z(f)\neq \emptyset$, then we have
$$ \dim(\overline{V} \cap \mathbb Z(f))=\dim{\overline{V}}-1.$$
\end{lemma}

We are  ready to prove Proposition~\ref{Pro2}.
It is not hard to see that $\overline{\cH}_k^1$,  $k=2,3, \ldots, d-1$, is absolutely irreducible, because  it is the same as the absolute irreducibility of the polynomial
$$ F(x_1,x_2,  \ldots, x_{k+1})= x_1^2 + x_2^2+ \ldots + x_k^2-x_{k+1}^2.$$
 Assume $F = RH$. Clearly, we see $deg_{x_1} R = deg_{x_1} H = 1$.
 Write
$$ R =\left (x_1+g(x_2,...,x_{k+1})\right)\left(x_1+h(x_2,...,x_{k+1})\right).$$
 We see that we should have $g = -h$ and so 
 $x_2^2+ \ldots + x_k^2-x_{k+1}^2 = -g^2$,
 which is easy to rule out for $k \ge 2$ (for example, by specializing
 $x_3=\cdots= x_k = 0, x_{k+1} = 1)$.

Since $\overline{\cH}_k^1$,  $k=2,3, \ldots, d-1$, is absolutely irreducible, it follows from the Affine Jacobian criterion that $\dim\overline{\cH}_k^1=d-1$ (see Lemma~\ref{Jac} below).
Notice that this completes the proof of Proposition~\ref{Pro2} in the case when $k=d-1$ with $k\geq 2$.
Thus, we may assume that $ d-k\ge 2$.
Observe by induction that  Proposition~\ref{Pro2} is a direct result from the following 
statement.

\begin{lemma}\label{irreducible} Assume that the characteristic of $\mathbb F_q$ is sufficiently large. Let $n\in \{1,2, \ldots, (d-k-1)\}$ with $k=2,3, \ldots, d-2$.
Suppose that $\bigcap_{j=1}^n \overline{\cH}_k^j$ is absolutely irreducible with dimension $d-n$.
Then $\bigcap_{j=1}^{n+1} \overline{\cH}_k^j$ is also absolutely irreducible with dimension $d-n-1$.
\end{lemma}

\begin{proof} From Lemma~\ref{claim1} and Lemma~\ref{dimlem}, it is clear that 
\begin{equation}\label{dimH}\dim \bigcap_{j=1}^{n+1} \overline{\cH}_k^j=d-n-1.\end{equation}
Thus, it remains to prove that $\bigcap_{j=1}^{n+1} \overline{\cH}_k^j$ is absolutely irreducible.
Assume that $\bigcap_{j=1}^{n+1} \overline{\cH}_k^j$ has $\nu$ absolutely irreducible components in 
$\overline{\mathbb F}_q$. By~\eqref{dimH} and Lemma~\ref{lem:LW} 
 to show that $\nu=1$, it is 
enough to prove that 
\begin{equation}
\label{eq:N_k}N(k,n)= \left| \bigcap_{j=1}^{n+1} \cH_k^j\right|=
 (1 + o(1))q^{d-n-1},
\end{equation}
where $\cH_k^j=\{x\in \mathbb F_q^d: x_1^{j+1}+ \cdots+x_k^{j+1} - x_{k+j}^{j+1}=0\}$.   Notice that $N(k,n)$ is the number of common solutions in $\mathbb F_q^d$ of the following equations
\begin{equation*} 
\begin{array}{ll} & x_1^2+x_2^2 +\cdots+ x_k^2 -x_{k+1}^2=0,\\
&\qquad \qquad \qquad \qquad \vdots\\
&  x_1^{j+1}+x_2^{j+1} +\cdots+ x_k^{j+1} -x_{k+j}^{j+1}=0,\\
&\qquad \qquad \qquad \qquad \vdots\\
 & x_1^{n+2}+x_2^{n+2} +\cdots+ x_k^{n+2} -x_{k+n+1}^{n+2}=0.\\
\end{array}, \end{equation*}

For each $j=1,2,\ldots, n+1$, define
$$ N_j(x_1,x_2,\ldots, x_k)=|\{x_{k+j}\in \mathbb F_q: x_{k+j}^{j+1}=x_1^{j+1}+\cdots+x_k^{j+1} \}|.$$
Since $x_{k+n+2}, \ldots, x_d\in \mathbb F_q$ are free variables and $x_{k+1}, \ldots, x_{k+n+1}\in \mathbb F_q$ depend only on $x_1,\ldots, x_k$, we can write 
$$N(k,n)=\sum_{x_1,\ldots,x_k\in \mathbb F_q} \left(\prod_{j=1}^{n+1} N_j(x_1,\ldots,x_k) \right) q^{d-k-n-1}.$$
In order to prove~\eqref{eq:N_k}, it therefore suffices to show that
\begin{equation}\label{hard} 
\sum_{x_1,\ldots,x_k\in \mathbb F_q} \left(\prod_{j=1}^{n+1} N_j(x_1,\ldots,x_k) \right)=(1+o(1)) q^k.
\end{equation}
For each $j=1,2, \ldots, n+1$, let $d_j=\gcd(j+1, q-1)$ and denote by $\chi_j$ the multiplicative character of order $d_j$. Then, from the orthogonality of 
multiplicative characters it follows that
$$
N_j(x_1, \ldots,x_k) = \sum_{i_j=0}^{d_j-1} \chi_j^{i_j}(x_1^{j+1} +\cdots+ x_k^{j+1}),
$$
see~\cite[Section~3.1]{IwKow04}.
Hence, the left hand side of~\eqref{hard} is written by
\begin{align*}
\sum_{x_1,\ldots,x_k\in \mathbb F_q}& \(\prod_{j=1}^{n+1} N_j(x_1,\ldots,x_k) \) 
 \\
& =\sum_{i_1=0}^{d_1-1} \ldots \sum_{i_{n+1}=0}^{d_{n+1}-1} \sum_{x_1,\ldots,x_k\in \mathbb F_q}  \prod_{j=1}^{n+1}
\chi_{j}^{i_{j}}(x_1^{j+1} +\cdots+x_k^{j+1}). 
\end{align*}
When $(i_1, \ldots i_{n+1})=(0,\ldots,0)$, the sum over $x_1,\ldots, x_k$ is $q^k$, where we use the usual convention that $\chi_0(0)=1$ for the trivial multiplicative character 
$\chi_0$.
Thus, to establish~\eqref{hard}, it is enough to prove that for each $(i_1, \ldots, i_{n+1})\ne (0, \ldots, 0)$ with $i_j=0,1,\ldots, d_j-1$,
$$\sum_{x_2,\ldots,x_k\in \mathbb F_q}\left|\sum_{x_1\in \mathbb F_q} \chi_1^{i_1}(x_1^2+\cdots+x_k^2)\cdots
\chi_{n+1}^{i_{n+1}}(x_1^{n+2} +\cdots+x_k^{n+2})\right|=o(q^k).$$

Now, we define the sets $B, G \subseteq \mathbb F_q^{k-1}$ (of `bad' and `good'
vectors  $(x_2,\ldots, x_k) \in \mathbb F_q^{k-1}$) by
\begin{equation*}
\begin{split}
B=\{(x_2,\ldots, x_k)& \in \mathbb F_q^{k-1}~:\\
&x_2^{j+1}+\cdots+ x_k^{j+1}=0\quad\mbox{for some}~~j=1,2,\ldots, n+1\},
\end{split}
\end{equation*} 
and
$$G=\mathbb F_q^{k-1}\setminus B.$$
For each fixed $\overline x=(x_2,\ldots, x_k)\in G$, define 
$$A_{j}(\overline x)=x_2^{j+1}+\cdots+x_k^{j+1} \quad \mbox{for}~~j=1,2, \ldots, n+1.
$$
Note that $A_{j}(\overline x)\ne 0$ for  $\overline x \in G$ and all $j=1,2, \ldots, n+1$. 

Since $|B|\leq (n+1)(n+2) q^{k-2}$,  it suffices to prove that for each $(i_1, \ldots, i_{n+1})\ne (0, \ldots, 0)$  with $i_j=0,1,\ldots, d_j-1$,
\begin{equation}
\label{End}
\begin{split}
\sum_{\overline x =(x_2,\ldots,x_k)\in G }&\left|\sum_{x_1\in \mathbb F_q} \prod_{j=1}^{n+1}\chi_j^{i_j}(x_1^{j+1}+A_{j}(\overline x))\right|\\
&=
 \sum_{\overline x =(x_2,\ldots,x_k)\in G }\left|\sum_{t\in \mathbb F_q} \prod_{j=1}^{n+1}\chi_j^{i_j}(t^{j+1}+A_{j}(\overline x))\right| =o(q^k). 
\end{split}
\end{equation}   
From the definition of $\chi_j$ and the fact that 
$(i_1,\ldots, i_{n+1}) \ne (0, \ldots, 0)$, notice that  $\chi_j^{i_j}$ is a nontrivial character  for some $j=1,\ldots, n+1$. If $\chi_j^{i_j}$ is a trivial character, then the term $\chi_j^{i_j}(t^{j+1}+A_{j}(\overline x))$ can be replaced by 1. Thus,  it suffices to prove~\eqref{End} under the  assumption that all $\chi_j^{i_j}$ are not-trivial characters.

We now consider the cases  $k=2$  and $k \ge 3$ separately. 

\begin{itemize}

\item  If  $k=2$, we must show that 
$$ \sum_{a\in G}\left|\sum_{t\in \mathbb F_q} \prod_{j=1}^{n+1}\chi_j^{i_j}(t^{j+1}+a^{j+1})\right|   =o(q^2),  $$                                         
for all $(i_1,\ldots, i_{n+1}) \ne (0, \ldots, 0)$  with $i_j=0,1,\ldots, d_j-1$.
Recall that if $a\in G$, then $a\ne 0$, thus
$$
\left|\sum_{t\in \mathbb F_q} \prod_{j=1}^{n+1}\chi_j^{i_j}(t^{j+1}+a^{j+1})\right|
= \left|\sum_{t\in \mathbb F_q} \prod_{j=1}^{n+1}\chi_j^{i_j}(t^{j+1}+1)\right|
$$
and  recalling Lemma~\ref{endlem3}
we obtain the desired estimate.

\item If $k\ge 3$, then it is easy to show that if the
characteristic of $\mathbb F_q$ is sufficiently large,
then for all but $O(q^{k-2})$ choices of 
$\overline x =(x_2, \ldots, x_k) \in G $, the polynomials 
$$
t^{j+1}+A_{j}(\overline x) \in {\mathbb F}_q[t], 
\qquad j=1,  \ldots, n+1
$$
have no pairwise common roots.
Indeed, assume $k\geq 3$.
Let $i_1, i_2\in \{2,3,\ldots, n+2\}$ with $i_1\ne i_2$. Notice that  if $t^{i_1}-A$ and $t^{i_2} - B$ have a common root then
$A^{i_2} = B^{i_1}$.
 For our expressions for $A$ and $B$ in $x_2, ..., x_k$
(assuming that $k\ge 3$) one can easily show that this leads to a
nontrivial equation and thus has $O(q^{k-2})$ solutions. 
For such $(x_2, \ldots, x_k)\in G$, 
the inner sum over $t\in \mathbb F_q$ in~\eqref{End} is
trivially estimated as $q$, and
for the remaining 
choices of $(x_2, \ldots, x_k)\in G$, we apply Lemmas~\ref{endlem1} and~\ref{endlem2}
to estimate the inner sum over $t\in \mathbb F_q$ in~\eqref{End}.
In conclusion, the left hand side of~\eqref{End} is bounded by $O(q^{k-1/2})$.
\end{itemize}

This establishes~\eqref{End} for every $k \ge 2$ and concludes the proof.  
\end{proof}

As we have mentioned, Lemma~\ref{irreducible} implies Proposition~\ref{Pro2}.

\subsection{Proof of Proposition~\ref{Pro3}} 

  The proof is based on the  Affine Jacobian criterion below 
 (see~\cite[Proposition~4.4.8]{Ga02}).
\begin{lemma} \label{Jac}Let $\overline{V} \subseteq\mathbb A^d$ be an irreducible algebraic set given by a system of $s$-polynomial equations $g_j(x)=0$, $j=1,2,\ldots,s$.
Suppose that $v\in \overline{V}$. Then $\overline{V}$ is smooth at $v$ if and only if   the rank of the $s\times d$ Jacobian matrix satisfies
$$
\rank \left[\frac{\partial{g_j}}{\partial{x_i}}(v)\right]_{s\times d}
\ge d-\dim\overline{V}.
$$
\end{lemma}

Now, since $d-k$ is the number of the polynomials $h_j$ in~\eqref{system} defining $\overline{\cH}_k$ which is absolutely irreducible with dimension $k$ by Proposition~\ref{Pro2},  we see from Lemma~\ref{Jac} that
$\overline{\cH}_k$ is smooth away from the origin if and only if 
$$ \rank\left[ \frac{\partial{h_j}}{\partial{x_i}}(x)\right]_{(d-k)\times d}=d-k$$
for all $x\in \overline{\cH}_k \setminus\{(0,\ldots,0)\}$.
Since we have assumed that the characteristic of $\mathbb F_q$ is sufficiently large,  it is clear from the Gauss elimination that 
$$\rank\left[ \frac{\partial{h_j}}{\partial{x_i}}(x)\right]_{(d-k)\times d}= \rank J_{k,d}(x),$$
where $J_{k,d}(x)$ denotes the $(d-k)\times d$ matrix given by the concatenation
$$J_{k,d}(x)  = [W_{k,d}(x)\, \| \, D_{k,d}(x) ]
$$
of the Wandermonde      
$$W_{k,d}(x)  = \left[ \begin{matrix}
x_i^j \end {matrix}\right]_{(d-k)\times k},
$$ 
and the diagonal matrix      
$$D_{k,d}(x)  = 
\left[ \begin{matrix}
-x_{k+1}& 0&  0&  \cdots&0\\
 0& -x_{k+2}^2&   0&      \cdots&0\\
 \vdots& \vdots& \vdots&   \vdots&  \vdots \\
 0& 0 &  0&  0&   -x_{d}^{d-k}
         \end {matrix}\right]. $$

In order to complete the proof of Proposition~\ref{Pro3},  it therefore suffices to prove the following two statements: 
\begin{itemize}
\item[{\bf (A1)}] if $d-k\geq 4$ and $k\geq 2$, then there exists $x\in \overline{\cH}_k\setminus \{(0,\ldots,0)\}$ with $\rank J_{k,d}(x) < d-k$;
\item[{\bf (A2)}]   if $d-k=1,2,3$ and $k\geq 2$, then $\rank J_{k,d}(x) = d-k$ for all $x \in\overline{\cH}_k\setminus \{(0,\ldots,0)\}$. 
\end{itemize}

First, let us prove {\bf(A1)}. Suppose that $k\geq 2$ is an even integer.
For each $l=1,2,\ldots, d-k$, choose an $\alpha_l\in \overline{\mathbb F}_q$ with $\alpha_l^{l+1}=k\cdot 1$, and
define
$$
x_{k+l}=\left\{\begin{array}{ll} 0\quad&\text{for}~~l\ge2~\text{even,}\\
                \alpha_l\quad&\text{for}~~l\ge 1~\text{odd}.\end{array}\right.
$$
Letting $x=(1,-1,\ldots,1,-1, x_{k+1}, \ldots, x_d)$, it is easy to check that $x\in \overline{\cH}_k\setminus\{(0,\ldots,0)\}$.
Since $d-k\geq 4$,  the matrix $J_{k,d}(x)$ has at least four rows, and its second row and fourth row are exactly same.
Thus, the rank of $J_{k,d}(x)$ must be less than $d-k$.
Next, assume that $k\geq 3$ is an odd integer. For each $l=1,2,\ldots, d-k$, select a $\beta_l\in \overline{\mathbb F}_q$ with $\beta_l^{l+1}=(k-1)\cdot1$, and
define
$$ x_{k+l}=\left\{\begin{array}{ll} 0\quad&\text{for}~~l\ge0~\text{even,}\\
            \beta_l\quad&\text{for}~~l\ge 1~\text{odd}.\end{array}\right.
$$
Taking $x=(1,-1,\ldots,1,-1,0, x_{k+1}, \ldots, x_d)$, 
we also see that $x\in \overline{\cH}_k\setminus\{(0,\ldots,0)\}$,
and the second row and the fourth row of $J_{k,d}(x)$ are  same.
Thus, the rank of $J_{k,d}(x)$ is less than $d-k$, which completes the proof of the statement ${\bf(A1)}$.

Now, we prove the statement {\bf (A2)}.
If $d-k=1$ and $k\geq2$, then  {\bf (A2)} is clearly true, because 
$$
J_{k,d}(x)=\left[ \begin{matrix}x_1 & x_2 & \cdots 
& x_{d-1} & -x_{d}\end{matrix}\right] 
\neq \left[ \begin{matrix}0&0\cdots & 0 &0\end{matrix}\right]
$$ 
for $x\neq (0,\ldots,0)$.
Assume that  $d-k=2$ and $k\geq2$. We must show that 
$$ \rank J_{k,d}(x) =\rank\left[ \begin{matrix}
x_1 &  x_2 & \cdots & x_{d-2}&  -x_{d-1}& 0\\
x_1^2 &   x_2^2&\cdots &  x_{d-2}^2 &   0& -x_{d}^2 \end{matrix}\right]=2
$$
for all $x\in \overline{\cH}_k\setminus \{(0,\ldots,0)\}$, 
where 
$$\overline{\cH}_k=\{x\in \mathbb A^d~:~x_1^2+\cdots+x_{d-2}^2-x_{d-1}^2 = x_1^3+\cdots+x_{d-2}^3-x_d^3=0\}.
$$
Notice that if $x=(x_1,\ldots,x_d)\in  \overline{\cH}_k\setminus \{(0,\ldots,0)\}$, then $x_j\neq 0$ for some $j=1,2,\ldots, d-2$.
Without loss of generality, we therefore assume that $x_1\neq 0$. Letting $u_j=x_j/x_1$ for $j=2,3,\ldots, d$, it is enough to show that
\begin{equation}\label{mainrank} \rank\left[ \begin{matrix}
1 &  u_2 & \cdots & u_{d-2}&  -u_{d-1}& 0\\
1 &   u_2^2&\cdots &  u_{d-2}^2 &   0& -u_{d}^2 \end{matrix}\right]=2,\end{equation}
where $(u_2,u_3,\ldots, u_d)\in \mathbb A^{d-1}$ satisfies
\begin{equation}\label{equ2}\left\{\begin{matrix}
 1+u_2^2+\cdots+u_{d-2}^2-u_{d-1}^2 &=0\\
 1+u_2^3+\cdots+u_{d-2}^3-u_{d}^3&=0 \end{matrix}\right.. \end{equation}
Notice that if $u_{d-1}, u_{d}\neq 0$,  then~\eqref{mainrank}
 holds, because 
 $$\rank \left[\begin{matrix} -u_{d-1}& 0\\ 0& -u_d^2\end{matrix}\right]=2.
 $$  
If $u_j=0$ or $1$ for all $j=2,3,\ldots, d-2$, then we see from~\eqref{equ2} that $u_{d-1}, u_d\neq 0$ and so there is nothing to prove.
On the other hand, if  $u_j\neq 0, 1$ for some $j=2,3, \ldots,(d-2)$ then
 $$\det \left[\begin{array}{cc} 1& u_j\\ 1& u_j^2\end{array}\right]=u_j(u_j-1) \neq 0 \quad\text{or}\quad
\rank \left[\begin{array}{cc} 1& u_j\\ 1& u_j^2\end{array}\right]=2.$$
Thus~\eqref{mainrank} is also true and we completes the proof of the statement {\bf(A2)} in the case when $d-k=2$ and $k\geq2 $.
Finally let us prove the statement {\bf(A2)}  when $d-k=3$ and $k\geq 2$.
Following the previous arguments,  our task is to show that
\begin{equation}\label{mainrank1} \rank\left[ \begin{array}{ccccccc}
1 &  u_2 & \cdots & u_{d-3}&  -u_{d-2}& 0&0\\
1 &   u_2^2&\cdots &  u_{d-3}^2 &   0& -u_{d-1}^2&0\\
1 &   u_2^3&\cdots &  u_{d-3}^3&0 &   0& -u_{d}^3 \end{array}\right]=3,\end{equation}
where $(u_2,u_3,\ldots, u_d)\in \mathbb A^{d-1}$ satisfies
\begin{equation}\label{equ3}\left\{\begin{array}{ll}
 1+u_2^2+\cdots+u_{d-3}^2-u_{d-2}^2 &=0,\\
 1+u_2^3+\cdots+u_{d-3}^3-u_{d-1}^3&=0,\\
 1+u_2^4+\cdots+u_{d-3}^4-u_{d}^4&=0. \end{array}\right. \end{equation}
 
{\bf Case~1}: Suppose that $u_j=0$ or 1 for all $j=2,3,\ldots, d-3$.
Then it follows from~\eqref{equ3} that $u_{d-2}, u_{d-1}, u_d \neq 0$, which implies that
$$\det\left[ \begin{array}{ccc}
  -u_{d-2}&  0&0\\
  0 &  -u_{d-1}^2 &   0\\
   0&  0 &   -u_{d}^3 \end{array}\right]\neq 0$$
hence
$$
\rank\left[ \begin{array}{ccc}
  -u_{d-2}&  0&0\\
  0 &  u_{d-1}^2 &   0\\
   0&  0 &   -u_{d}^3 \end{array}\right]=3. $$
Thus~\eqref{mainrank1} also follows.

{\bf Case 2}: Suppose that $u_i, u_j\neq 0, 1$ with $u_i\neq u_j$ for some $i,j=2,3,\ldots, d-3$.
Then it follows that
$$\det\left[\begin{matrix} 1&u_i&u_j\\
                                        1&u_i^2&u_j^2\\
                                         1&u_i^3&u_j^3\end{matrix}\right]=u_iu_j(u_i-1)(u_j-1)(u_j-u_i)\neq 0.$$
Thus 
$$\rank\left[\begin{matrix} 1&u_i&u_j\\
                                        1&u_i^2&u_j^2\\
                                         1&u_i^3&u_j^3\end{matrix}\right]=3,
$$
which implies~\eqref{mainrank1}.

{\bf Case 3}: Suppose that $u_j\neq 0,1$ for some $j=2,3,\ldots, d-3$, and $u_i=u_j$ if $u_i\neq 0,1$ for $i=2,3,\ldots, d-3$. Let 
\begin{equation*}
\begin{split}
a&=|\{l\in \{2,3,\ldots, d-3\}~:~u_l=1\}|, \\ 
b&=|\{l\in \{2,3,\ldots, d-3\}~:~u_l\neq 0,1\}|.
\end{split}
\end{equation*}
Then $a\geq 0$ and $b\geq 1$ are integers.
Thus~\eqref{equ3} is same as
\begin{equation}\label{equ4} \begin{array}{ll}
(1+a)1+bu_j^2-u_{d-2}^2&=0,\\
(1+a)1+bu_j^3-u_{d-1}^3&=0,\\
(1+a)1+bu_j^4-u_d^4&=0,\end{array}
\end{equation}
where $(1+a)\in \mathbb N$. 
 
We now claim that either $u_{d-2}\neq 0$ or $u_d\neq 0$. To see this, assume that
$u_{d-2}, u_d=0$. Then from~\eqref{equ4} we see that $(1+a)1+bu_j^2=0=(1+a)1+bu_j^4$.
This implies that $b u_j^2(u_j-1)(u_j+1)=0$. Thus we conclude that $u_j=-1$, because $u_j\neq 0,1$ and $b\in \{1, \ldots, d-4\}$ (assuming that the
characteristic of $\mathbb F_q$ is sufficiently large). However, since $(1+a)1+b(-1)^2= (1+a+b)1\neq 0$,  it is impossible that $u_j=-1$ (again assuming that the
characteristic of $\mathbb F_q$ is sufficiently large) and the claim is justified.

If $u_{d-2}\neq 0$, then~\eqref{mainrank1} follows by the observation that
$$\det\left[\begin{array}{ccc} 1&u_j&- u_{d-2}\\
                                         1&u_j^2& 0\\
                                         1 &u_j^3&0\end{array} \right]=-u_{d-2}u_j^2(u_j-1)\neq 0.$$
On the other hand, if $u_d\neq 0$, then~\eqref{mainrank1} also follows by the observation that
$$\det\left[\begin{array}{ccc} 1&u_j&0\\
                                         1&u_j^2& 0\\
                                         1 &u_j^3&-u_d^3\end{array} \right]=-u_{d}^3u_j(u_j-1)\neq 0.$$
By  {\bf Case 1}, {\bf Case 2} and {\bf Case 3}, we establish the statement {\bf (A2)} in the case when $d-k=3$ and $k\geq 2$.
This concludes the proof of  Proposition~\ref{Pro3}.

\section*{Acknowledgements}

The authors are grateful to Anthony Flatters and Tom Ward for useful discussions
and in particular for the idea of the proof of Lemma~\ref{endlem15}. 

The first author was supported by  the research grant of the Chungbuk National University in 2012 and  Basic Science Research Program through the National
Research Foundation of Korea  funded by the Ministry of Education, Science and
Technology (2012010487). The third author was supported  by the  research grant of the 
Australian Research Council (DP130100237).

\end{document}